\documentclass[leqno,12pt]{amsart}
\usepackage{amsfonts}
\usepackage{amsmath,amssymb,amsthm}

\newtheorem{theorem}{Theorem}[section]
\newtheorem{proposition}[theorem]{Proposition}
\newtheorem{lemma}[theorem]{Lemma}
\newtheorem{corry}[theorem]{Corollary}

\newtheorem{example}[theorem]{Example}
\newtheorem{rem}[theorem]{Remark}

\numberwithin{claim}{theorem}

\renewenvironment{proof}{\textit{Proof.}}{\hfill\ensuremath{\qed}}

\newcommand{\tr}{\mathrm{tr}}
\def \qed{\hfill{\hbox{$\square$}}}

\numberwithin{equation}{section}

\begin{document}
\title[Biconservative surfaces in  $\mathbb E^n$]{On biconservative surfaces in Euclidean spaces }
\author{R\"uya Ye\u g\.{I}n \c{S}en}
\address{Istanbul Medeniyet University, Faculty of Engineering and Natural Sciences, Department of Mathematics, 34700 Uskudar, Istanbul/Turkey}
\email{ruya.yegin@medeniyet.edu.tr}
\author{Nurett\.{I}n Cenk Turgay}
\address{Istanbul Technical University, Faculty of Science and Letters, Department of Mathematics, 34469 Maslak, Istanbul/Turkey}
\email{turgayn@itu.edu.tr}

\begin{abstract}
In this paper, we study biconservative surfaces with 
parallel normalized mean curvature vector in $\mathbb{E}^4$.
We obtain complete local classification in $\mathbb{E}^4$ 
for a biconservative PNMCV surface. We also give an example 
to show the existence of PNMCV biconservative surfaces 
in $\mathbb{E}^4$. 
 \end{abstract}


\maketitle
\section{Introduction}
Let $(M^m,g)$ and $(N^n,\tilde g)$ be some Riemannian manifolds. Then,  the bi-energy functional is defined by 
\begin{equation}\nonumber
E_2(\psi)=\frac{1}{2}\int_{M}|\tau(\psi)|^2v_g\quad 
\end{equation}
{ whenever $\psi:M\rightarrow N$   is a smooth mapping,} where $\tau(\psi)$ denote the tension field of $\psi $.

A mapping $\psi:M\rightarrow N$  is said to be biharmonic if it is a critical point of $E_2$. 
In \cite{Ji2} it was proved that  mapping  $\psi$ is biharmonic if and only if it satisfies the Euler-Lagrange equation associated with this bi-energy functional given by
\begin{equation}\label{EulerLagrangeEq}
\tau_2(\psi)=0,
\end{equation} 
where  $\tau_2$ is the bitension field defined by $\tau_2(\psi)=\Delta\tau(\psi)-tr\tilde R(d\psi,\tau(\psi))d\psi$ (See also \cite{Ji}).

In particular, if $\psi$ is an isometric immersion, then $M$ is said to be a biharmonic submanifold of $N$.
In  this case,  by considering tangential and normal components of $\tau_2(\psi)$, one can obtain the following proposition.
\begin{proposition}\label{biharmonic}
Let $x:M^m\rightarrow N^n$ be an isometric immersion 
between two Riemannian manifolds. Then, $x$ is biharmonic 
if and only if the equations
\begin{equation}\label{tangent component}
m \mathrm{grad} \left\|  H\right\| ^2+ 4 \mathrm{trace}   A_{\nabla^\perp_{\cdot}H} (\cdot)+4\mathrm{trace} (\tilde{R}(\cdot,H)\cdot)^T=0
\end{equation}
and 
\begin{equation}\label{normal component}
\mathrm{trace} \alpha_\psi(A_H(\cdot),\cdot)-\Delta^\perp H+\mathrm{trace} (\tilde{R}(\cdot,H)\cdot)^\perp=0
\end{equation}
are satisfied, where $A$, $H$ and $\alpha_\psi$ denote the shape operator, the mean 
curvature vector and second fundamental form of $\psi$,  
$\nabla^\perp$ is the normal connection of $M$ and  $\Delta^\perp$ is the Laplacian associated with $\nabla^\perp$.
\end{proposition}

From Proposition \ref{biharmonic}, one can see that an isometric immersion $x:M\rightarrow N$ is biharmonic if its mean curvature $H$ vanishes identically.   In \cite{ChenOpenProblems},  Bang-Yen Chen conjectured that the converse of this statement is also true if the ambient space is Euclidean. Chen's biharmonic conjecture has been verified in a lot of particular cases so far (see for example \cite{ChenRapor,ChenMunt1998Bih,Dimitric1992,Hasanis-Vlachos,YuFu2014ThreeD,MaetaGlobal}). However, the conjecture is still open.

On the other hand, a mapping $\psi:M\rightarrow N$ satisfying the condition
\begin{equation}\label{BiconsEq}
\langle \tau_2(\psi), d\psi\rangle=0,
\end{equation} 
that is weaker than \eqref{EulerLagrangeEq}, is said to be biconservative. In particular, if $\psi=x $ is an isometric immersion, then   \eqref{BiconsEq} is equivalent to
\begin{equation}\nonumber
\tau_2(x)^T=0,
\end{equation}
where $\tau_2(x)^T$ denotes the tangential part of  $\tau_2(x)$. In this case,  $M$ is said to be a biconservative submanifold of $N$. Before we proceed, we would like to note that one can conclude the following well-known proposition,  by considering Proposition \ref{biharmonic}(See for example \cite{CMOP}).

\begin{proposition}
Let  $x:M^m\rightarrow N^n$ be an isometric immersion between two Riemannian manifolds. Then,  $x$ is biconservative if and only if the equation \eqref{tangent component} is satisfied.
\end{proposition}

In order to understand geometry  of biharmonic submanifolds, biconservative  immersions have been studied in many papers so far, For example,  biconservative immersions into pseudo-Euclidean spaces were studied in \cite{YuFu2014Lor3S,YFuTurgay,MOR20141, TurgayHHypers,UpadhyayTurgay}. On the other hand, in \cite{MOR2016BicSurf}, Montaldo  et al. study biconservative surfaces in four-dimensional space form with constant mean curvature. In \cite{FOLSnRHnR}, the complete classification of surfaces in product spaces $\mathbb S^n\times\mathbb R$ and
$\mathbb H^n\times\mathbb R$ with parallel mean curvature vector was obtained by D. Fetcu, C. Oniciuc, and A. L. Pinheiro.

In this paper, we consider biconservative surfaces in Euclidean spaces. In Sect. 2, after we describe the notation that we will use, we give basic facts on  biconservative submanifolds. In Sect. 3, we give complete classification of biconservative surfaces in Euclidean spaces with paralel normalized mean curvature vector field.

\section{Biconservative submanifolds in Euclidean spaces}
Let $\mathbb E^n$ denote the Euclidean $n$-space with the canonical positive definite Euclidean metric tensor given by  
$$
\widetilde g=\langle\ ,\ \rangle=\sum\limits_{i=1}^n dx_i^2,
$$
where $(x_1, x_2, \hdots, x_n)$  is a rectangular coordinate system in $\mathbb E^n$ and $\tilde\nabla$ stands for its Levi-Civita connection.

Let $M$ be an $m$-dimensional submanifold of $\mathbb E^n$ and $\nabla^\perp$ denote its normal connection. A normal vector field $\xi$ is called parallel if
\begin{equation}\nonumber
\nabla^\perp_X \xi=0
\end{equation}
whenever $X$ is tangent to $M$. On the other hand, the mean curvature vector field $H$ of $M$ is defined by 
\begin{equation}\label{mean curvature}
H=\frac{1}{2}\tr h(\cdot,\cdot),
\end{equation}
where $h$ is the second fundamental form of $M$. Assume that the mean curvature (function) of $M$ given by
\begin{equation}\nonumber
f=\langle H,H\rangle^{1/2}
\end{equation}
is a non-vanishing function. In this case, if the unit normal vector field along the mean curvature vector field $H$ of $M$ is parallel, then $M$ is said to have parallel normalized mean curvature vector field and called a PNMCV submanifold of $\mathbb E^n$. It is obvious that  PNMCV submanifolds generalize non-minimal submanifolds with parallel mean curvature vector and a PNMCV submanifolds has parallel mean curvature vector if and only if it has constant mean curvature, i.e., $f$ is constant. It is possible 
to find examples of PNMCV submanifolds with non-constant mean curvature (See for example \cite{Chen1980,Leung}).

Since the curvature tensor $\tilde R$ of $\mathbb E^n$ vanishes identically, the following proposition is obtained immediately from \eqref{tangent component}.
\begin{proposition}\label{BicPNMCEnProp}
Let $M$ be an $m$-dimensional PNMCV submanifold of the Euclidean space $\mathbb E^n$. Then, $M$ is biconservative if and only if
\begin{equation}\label{BicPNMCEn}
A_{m+1} (\mathrm{grad}f)=-\frac{mf}2\mathrm{grad}f,
\end{equation}
where $e_{m+1}$ is the parallel mean curvature vector field and $A_{m+1}$ is the shape operator along $e_{m+1}$.
\end{proposition}

\begin{rem}
If the mean curvature vector of $M$ is parallel, then \eqref{BicPNMCEn} is satisfied trivially. Therefore, after this point we will assume that $\mathrm{grad}f$ does not vanish at any point of $M$.
\end{rem}


\subsection{Basic equations in the theory of surfaces of $\mathbf {E^n}$}
Let $M$ be a surface in $\mathbb E^n$, $\nabla, h$ and  $A$ denote its the Levi-Civita connection, second fundamental form and shape operator, respectively. Note that $R$ and $\tilde R$ will stand for curvature tensor of $M$ and $\mathbb E^n$, respectively.

For tangent vector fields $X,Y,Z$ on $M$ the Codazzi equation $(\tilde R(X,Y)Z)^\perp=0$ and the Gauss equation $(\tilde R(X,Y)Z)^T=0$ take the form
\begin{equation}\label{Codazzi}
(\bar{\nabla}_X h)(Y,Z)=(\bar{\nabla}_Y h)(X,Z)
\end{equation}
and
\begin{equation}\label{TheGausseq}
R(X,Y)Z=A_{h(Y,Z)}X-A_{h(X,Z)}Y,
\end{equation}
respectively,  where $(\bar{\nabla}_X h)(Y,Z)$ is defined by 
$$(\bar{\nabla}_X h)(Y,Z)=\nabla^\perp_{X}h(Y,Z)-h(\nabla_{X}Y,Z)-h(Y,\nabla_{X}Z).$$

The Ricci equation  $(\tilde R(X,Y)\xi)^\perp=0$ takes the form
\begin{equation}\label{ricci}
R^D(X,Y)\xi=h(X,A_\xi Y)-h(A_\xi X,Y)
\end{equation}
whenever $X,Y$ are tangent and $\xi$ is normal to $M$.

\section{Biconservative PNMCV Surfaces in $\mathbb{E}^4$}
In this section, we would like to obtain complete local classification of biconservative surfaces with parallel normalized 
mean curvature vector in the Euclidean 4-space $\mathbb{E}^4$.

First we would like to obtain shape operator and Levi-Civita connection of a biconservative PNMCV surface in $\mathbb{E}^4$. 
\begin{lemma}\label {Biconservative surfaces T.1.}
Let $M$ be a surface in $\mathbb{E}^4$ with non-vanishing mean curvature. Then, $M$ is a biconservative PNMCV surface if and only if there exists a local orthonormal frame field $\left\lbrace e_1,e_2;e_3,e_4\right\rbrace$ 
such that 
\begin{itemize}
\item[(1)] The Levi-Civita connection $\nabla$ and normal connection $\nabla^\perp$ of $M$ satisfy
\begin{subequations}\label{connectionsAll}
\begin{eqnarray}
\label{connectionsa}\nabla_{e_1}e_1&=&\nabla_{e_1}e_2=0,\\
\label{connectionsb}\nabla_{e_2}e_1&=&\frac{-3e_1(f)}{4f}e_2,\quad \nabla_{e_2}e_2=\frac{3e_1(f)}{4f}e_1,\\
\label{connectionsc}\nabla^\perp{e_3}&=&\nabla^\perp{e_4}=0,
\end{eqnarray}
\end{subequations} 

\item[(2)] Shape operators along $e_3$ and $e_4$ have 
matrix representations given by 
\begin{equation}\label{shape operator}
A_{e_3}=\left(\begin{array}{ccc}
-f&0\\
0&3f\\
\end{array}\right),\qquad 
A_{e_4}=\left(\begin{array}{ccc}
cf^{3/2}&0\\
0&- cf^{3/2}\\
\end{array}\right)
\end{equation}
for some constant $c$ and a smooth non-vanishing function $f$ such that $e_2(f)=0$ and  $e_1(f)\neq0$.
\end{itemize}
\end{lemma}

\begin{proof}
Let the mean curvature vector of $M$ be $H$ and $\langle H,H\rangle=f$. In order to prove the necessary condition, we assume that  $M$ is biconservative and  $H$ is parallel.   We choose a local frame field $\{e_3,e_4\}$ of the normal bundle of $M$ as $e_3=\frac{H}{f}$. Note that \eqref{BicPNMCEn} is satisfied for $m=2$. On the other hand, by the assumption $e_3$ is parallel. $e_4$ is also parallel because the co-dimension of $M$ is 2. Therefore, we have 
\begin{equation}\nonumber
\nabla^\perp_X{e_3}=\nabla^\perp_X{e_4}=0
\end{equation} 
for any tangent vector field $X$ which yields \eqref{connectionsc}. As $H$ is proportional to $e_3$, we have 
\begin{equation}\label{EqutrA3trA4}
\frac{1}{2}\tr A_3=f\quad\mbox{and}\quad \tr A_4=0.
\end{equation}

We choose a local frame field $\{e_1,e_2\}$ of the normal bundle of $M$ so that 
$A_3=\mathrm{diag}(h^3_ {11},h^3_ {22}).$ Then, because of \eqref{BicPNMCEn}, we may assume $e_1= \nabla f/\left\| \nabla f\right\| $ and $h^{3}_{11}=-f$. We will prove that the frame field  $\{e_1,e_2;e_3,e_4\}$ satisfies the other  conditions given in the lemma.

The first equation in \eqref{EqutrA3trA4} implies   $h^{3}_{22}=3f$. Since $e_3$ is parallel, the Ricci equation \eqref{ricci} for $X=e_1, Y=e_2$ and $\xi=e_3$ yields 
$$4fh(e_1,e_2)=0.$$
Therefore, we have $\langle A_4(e_1),e_2\rangle=0$. This equation and the second equation in \eqref{EqutrA3trA4} imply
$$
A_{e_3}=\left(\begin{array}{ccc}
-f&0\\
0&3f\\
\end{array}\right),\quad\mbox{and}\quad 
A_{e_4}=\left(\begin{array}{ccc}
\lambda&0\\
0&-\lambda\\
\end{array}\right)
$$
for a smooth function $\lambda$. 

By using Codazzi equation \eqref{Codazzi} for $X=e_1$, $Y=Z=e_2$, we obtain
\begin{equation}\label{3.1}
e_1(\lambda)=-2\lambda\omega_{12}(e_2)\quad \text{and}\quad e_1(f)=\frac{-4f}{3}\omega_{12}(e_2)
\end{equation}
which imply
\begin{equation}\label{E.1}
2fe_1(\lambda)=3\lambda e_1(f).
\end{equation}
The Codazzi equation \eqref{Codazzi} for $X=e_2$, $Y=Z=e_1$ gives
\begin{equation}\label{3.1b}
\omega_{12}(e_1)=0, \quad e_2(\lambda)=0.
\end{equation}
By integrating equation \eqref{E.1} and considering \eqref{3.1b}, we have $\lambda=cf^{3/2}$ for some constant $c$. Moreover,  since $e_1$ is proportional 
to $\nabla f$, $f$ satisfies $e_2(f)=0$ and  $e_1(f)\neq0$. Therefore, we have the condition (1) of the lemma. On the other hand, the second equation in \eqref{3.1} and the first equation in \eqref{3.1b} give the condition (2) of the lemma. Hence, we completed the proof of the necessary condition.

Conversely, let $M$ be a surface with a local orthonormal frame field $\{e_1,e_2;e_3,e_4\}$ satisfying the conditions given in the Lemma. Then, \eqref{shape operator} implies $H=fe_3$ and $\mathrm{grad} f=e_1(f)e_1+e_2(f)e_2=e_1(f)e_1$. Therefore, $\mathrm{grad} f$ is eigenvalue of the shape operator $A_{e_3}$ and $e_3$ is the normalized mean curvature vector field of $M$. Moreover, \eqref{connectionsc} yields that $e_3$ is parallel. Therefore, $M$ is a PNMCV surface. Moreover, Proposition \ref{BicPNMCEnProp} is satisfied which yields that $M$is biconservative. Hence, the proof of the sufficient condition is completed.
\end{proof}

As an immediate consequence of Lemma \ref{Biconservative surfaces T.1.}, we would like to state the following corollary. 
\begin{corry}
Let $M$ be a biconservative PNMCV surface in $\mathbb{E}^4$. Then, the Gaussian curvature $K$ and the mean curvature $f$ of $M$ satisfy $K=-3f^2-c^2f^3 $ for a constant  $c$.
\end{corry}

Let $M$ be a PNMCV surface in $\mathbb{E}^4$, $f$ its mean curvature and $m\in M$ with $f(m)\neq 0$ and $(\mathrm{grad}\, f)(m)\neq0$. Next, by using Lemma \ref{Biconservative surfaces T.1.}, we would like to construct a local coordinate system on a PNMCV biconservative surface $M$ in $ \mathbb{E}^4$ on a neighborhood of $m\in M$.
\begin{lemma}\label{principal directions}
Consider a local orthornormal frame field $\{e_1,e_2;e_3,e_4\}$ on $M$ given in  Lemma \ref{Biconservative surfaces T.1.}. Then there exists local coordinate system $\{s,t\}$ on a neighborhood of $m$ such that   
$$f=f(s),\qquad e_1=\frac{\partial}{\partial s},\quad\mbox{ and }\quad e_2=f(s)^{3/4}\frac{\partial}{\partial t}$$
\end{lemma}
\begin{proof}
Let  $\{e_1,e_2;e_3,e_4\}$ be a  a local orthornormal frame field given in  Lemma \ref{Biconservative surfaces T.1.}. Because of  $\nabla_{e_1}e_2=0, \nabla_{e_2}e_1=-\frac{3e_1(f)}{4f}e_2$, we
have $\left[e_1,e_2\right]=\frac{3e_1(f)}{4f}e_2$ which gives $\left[e_1,Ee_2 \right]=0$
for any function $E$ satisfying
\begin{equation}\label{Lemprincipal directionsEq1}
e_1(E)=\frac{-3e_1(f)}{4f}E.
\end{equation} 
Thus, there exists a local 
coordinate system $(s,t)$ such that   $e_1=\frac{\partial}{\partial s}$ and $E e_2=\frac{\partial}{\partial t}$.

Moreover, $e_2(f)=0$ which yields $f=f(s)$ and because of \eqref{Lemprincipal directionsEq1} we can choose $E$ as 
$E=f(s)^{-3/4}.$
\end{proof}

\begin{corry}\label{CorryMCODE}
Let $M$ be a biconservative surface with parallel normalized	mean curvature vector in $\mathbb{E}^4$.
Then, 
the mean curvature of $M$ satisfies the following partial differential equation 
\begin{equation} \label{mean curvature  partial differential}
\frac{9 f'(s)^2}{16 f(s)^{7/2}}+c^2 f(s)^{3/2}+9 f(s)^{1/2}=c_2^2
\end{equation}
for a positive constant $c_2^2$, where $s$ is the local coordinate given in Lemma \ref{principal directions}.
\end{corry}
\begin{proof}
Consider a local orthonormal frame field $\{e_1,e_2;e_3,e_4\}$ given in  Lemma \ref{Biconservative surfaces T.1.} and local coordinate systme $(s,t)$ given in Lemma \ref{principal directions}. Note that the Gauss equation \eqref{TheGausseq} for $X=Z=\partial_s$ and $Y=\partial_t$ gives
$$f(s) f''(s)-\frac{7}{4} f'(s)^2+4 f(s)^4+\frac{4}{3} c^2 f(s)^5=0.$$
By multiplying this equation with $\frac{9 f'(s)}{8 f(s)^{9/2}}$
and integrating the equation obtained, we get \eqref{mean curvature  partial differential} for a constant $c_2$ which can be assumed to be positive.
\end{proof}


\begin{proposition}\label{PNMCVBicPrflCrv}
Let $M$ be a proper PNMCV biconservative surfaces  in $\mathbb{E}^4$, where $f$ is the mean curvature of $M$ in $\mathbb{E}^4$ and $e_1=\frac{\nabla f}{\left| \nabla f\right| }$.	Then, 
\begin{itemize}
\item [(a)] An integral curve of $e_1$ lies on a $3$-dimensional
hyperplane of  $\mathbb{E}^4$.
\item [(b)] The curvature and torsion of an integral curve of $e_1$ are 
\begin{subequations}\label{KTAUALL}
\begin{eqnarray}
\label{KTAUALLa}\kappa(s)&=&f(s)\sqrt{1+c^2f(s)},\\
\label{KTAUALLb} \tau(s)&=&\frac{c{f}^\prime (s)}{2\sqrt{f(s)}(1+c^2f(s))}.
\end{eqnarray}
\end{subequations}
\item [(c)] Any two integral curves of $e_1$ are congruent.
\end{itemize} 
\end{proposition}

\begin{proof} Let $\left\lbrace e_1,e_2;e_3,e_4\right\rbrace  $ be the local orthonormal frame on $M$ given by Lemma \ref{Biconservative surfaces T.1.} and 
we suppose that $\gamma$ is an integral curve of $e_1$ and it is 
parametrized by $\gamma(s)=x(s,t_0)$. Let $T=\gamma^\prime$ be tangent of $M$ and consider the moving frame 
field span $\left\lbrace T(s),N(s),B(s),\tilde{B}(s)\right\rbrace $ of the curve $\gamma$ on 
$M$. We consider the following  Frenet formulas
\begin{eqnarray}\label{Frenet formulas}
\frac{dT}{ds}&=&\kappa N,\notag\\
\frac{dN}{ds}&=&-\kappa T+\tau B,\notag\\
\frac{dB}{ds}&=&-\tau N +\tau_2 \tilde{B},\\
\frac{d\tilde{B}}{ds}&=&-\tau_2 B,\notag
\end{eqnarray} 
where $\kappa,\tau$ and $\tau_2$  are curvatures of $\gamma$. 

We proceed to compute the curvature and torsion of an integral curve $\gamma$ of 
$e_1$. By combining \eqref{shape operator} with Gauss formula and considering Weingarten formula we obtain 
\begin{equation}\label{IntCrve1TT}
\tilde{\nabla}_{T(s)} T(s)=-f(s)e_3(s)+cf(s)^{3/2}e_4(s),
\end{equation}

\begin{equation}\label{IntCrve1Te3}
\tilde{\nabla}_{e_1}e_3=e_3^{\prime}(s)=f(s)T
\end{equation}  
 and 
\begin{equation}\label{IntCrve1Te4}
\tilde{\nabla}_{e_1}e_4=e_4^{\prime}(s)=-cf(s)^{3/2}T,
\end{equation} 
where $e_3(s), e_4(s)$ are restrictions of $e_3$ and $e_4$ to $\gamma$.

By combining \eqref{Frenet formulas} with \eqref{IntCrve1TT}, we get 
\eqref{KTAUALLa} and
\begin{align}\label{N(s)}
N(s)=\frac{dT}{ds} \frac{1}{\kappa}&= \dfrac{-1}{\sqrt{1+c^2f(s)}}e_3(s)+ c\sqrt{\frac{f(s)}{1+c^2f(s)}}e_4(s).
\end{align} 
By differentiation of \eqref{N(s)} with respect to $e_1$ and using \eqref{IntCrve1Te3}, \eqref{IntCrve1Te4}, we obtain
\begin{align}\label{frac{dN}{ds}}
\frac{dN}{ds}= \dfrac{c^2 f^\prime(s)}{2(1+c^2f(s))^{3/2}}e_3(s)
+\dfrac{c f^\prime(s)}{2\sqrt{f(s)}(1+c^2f(s))^{3/2}}e_4(s)
- \frac{f(s)(1+c^2)}{\sqrt{{1+c^2f(s)}}}T.
\end{align}
By combining equations \eqref{KTAUALLa}, \eqref{Frenet formulas} and \eqref{frac{dN}{ds}}, 
we obtain the torsion $\tau$ of $\gamma$ as given in \eqref{KTAUALLb}
and 
\begin{equation}\label{B(s)}
B(s)=\frac{c\sqrt{f(s)}}{\sqrt{1+c^2f(s)}}e_3(s)+\frac{1}{\sqrt{1+c^2f(s)}}e_4(s).
\end{equation}
By applying $e_1$ to the equation \eqref{B(s)} and using \eqref{IntCrve1Te3}, \eqref{IntCrve1Te4}, we obtain
$$\dfrac{dB}{ds}=\tau N(s),$$ 
which yields that $\tau_2=0$. Hence, we have the part (a) and part (b) of the Lemma.

Now, we want to show part (c) of the Lemma. Let $m_1,m_2\in M$ lie on the same integral curve of $e_2$. Consider  the local coordinate system $(s,t)$ given in Lemma \ref{principal directions}.  Note that $e_2(f)=0$ and $\nabla_{e_1}e_2=0$ yields $e_2(e_1(f))=0$ which implies $e_1(f)=f'(s)$.  Therefore, because of \eqref{KTAUALL}, any integral curves $\gamma_1$ and $\gamma_2$ of $e_1$ have the same curvature and torsion functions. Hence, they are congruent to each other. 
\end{proof}

Now, we are ready to get main classification theorem
\begin{theorem}\label{MainTHMparametrization}
Let $M$ be a  PNMCV surface in the $4$-dimensional 
Euclidean space $\mathbb{E}^4$ with a point $m\in M$ at which $f(m)>0, (\mathrm{grad} f)(m) \neq 0$, where $f$ is the mean curvature of $M$. If $M$ is biconservative, then there exists a neighborhood of $m$ on which $M$ is congruent to the simple rotational surface
\begin{equation}\label{MainTHMparametrizationSurfParam}
x(s,t)=\left(\alpha _1(s)\cos t, \alpha _1(s)\sin t,\alpha _2(s),\alpha _3(s)\right)
\end{equation}
with arc-length parametrized smooth profile curve $\displaystyle\alpha(s)=\left(\alpha _1(s),\alpha _2(s),\alpha _3(s)\right)$, 
$$\alpha _1(s)= \frac1{c_2 f(s)^{3/4}}$$
 whose curvature and torsion are given by \eqref{KTAUALL}. 
\end{theorem}

\begin{proof}
We consider a local orthonormal frame field $\{e_1,e_2;e_3,e_4\}$ given in Lemma \ref{Biconservative surfaces T.1.} with 
$$e_1=\frac{\partial}{\partial s}\quad\mbox{ and }\quad e_2=f(s)^{3/4}\frac{\partial}{\partial t},$$
where $(s,t)$ is local coordinate system give in Lemma \ref{principal directions}. Note that we also have $f=f(s)$ which satisfies
\eqref{mean curvature  partial differential} for a constant $c_2$ because of Corollary \ref{CorryMCODE}. Moreover, the induced metric of $M$ is
\begin{equation}\label{E4BicSurfPNMCVIndcMetric}
g=ds\otimes ds+\frac 1{f(s)^{3/2}}dt\otimes dt.
\end{equation}

By combining \eqref{connectionsb} with \eqref{shape operator}, we have 
\begin{subequations}\label{MainTHMparametrizationEq1ALL}
\begin{eqnarray}
\label{MainTHMparametrizationEq1a} \tilde{\nabla}_{\partial_{t}}\partial_{s}&=&-\frac{3f'}{4f}\partial_{t},\\
\label{MainTHMparametrizationEq1b} \tilde{\nabla}_{e_2}{e_2}&=&\frac{3f'}{4f}\partial_{s} +3fe_3- c f^{3/2} e_4,\\
\label{MainTHMparametrizationEq1c} \tilde{\nabla}_{\partial_{t}}{e_3}&=&-3f\partial_{t},\\
\label{MainTHMparametrizationEq1d} \tilde{\nabla}_{\partial_{t}}{e_4}&=&cf^{3/2}\partial_{t}.
\end{eqnarray}
\end{subequations}

Let $\text{x}:M\rightarrow\mathbb E^4$ be an isometric immersion. Then, \eqref{MainTHMparametrizationEq1a} becomes
$$\text{x}_{ts}=\frac{-3 f^\prime}{4f} \text{x}_t.$$
By solving this equation, we get
\begin{equation}\label{MainTHMparametrizationEq2}
\text{x}(s,t)=f^{-3/4}\Theta (t)+\Gamma (s)
\end{equation}
for some $\mathbb E^4$-valued smooth functions $\Theta, \Gamma$.

By combining \eqref{MainTHMparametrizationEq2} with \eqref{MainTHMparametrizationEq1b}, we obtain
$$f(s)^{3/4} \Theta ''(t)+\frac{9 f'(s)^2}{16 f(s)^{11/4}}\Theta (t)-\frac{3 f'(s) }{4 f(s)}\Gamma '(s)-3 f(s) e_3+c f(s)^{3/2} e_4=0.$$
By applying $\partial_t$ to this equation and using \eqref{MainTHMparametrizationEq1a}, \eqref{MainTHMparametrizationEq1c} and \eqref{MainTHMparametrizationEq1d} we obtain
$$\Theta'''(t)+\left(\frac{9 f'(s)^2}{16 f(s)^{7/2}}+c^2 f(s)^{3/2}+9 f(s)^{1/2}\right)\Theta '(t) =0.$$
By combining this equation with \eqref{mean curvature  partial differential}, we get
$$\Theta'''(t)+c_2^2\Theta '(t) =0.$$
Thus, $\Theta$ has the form
$\Theta(t)=\cos(c_2t)A_1+\sin(c_2t)A_2+A_3$
for some constant vectors $A_1,A_2,A_3$. Therefore, \eqref{MainTHMparametrizationEq2} becomes
\begin{equation}\label{MainTHMparametrizationEq3}
\text{x}(s,t)=f(s)^{-3/4}\cos(c_2t)A_1+f(s)^{-3/4}\sin(c_2t)A_2+f(s)^{-3/4}A_3+\Gamma (s).
\end{equation}

By considering \eqref{E4BicSurfPNMCVIndcMetric}, we obtain
\begin{align}\nonumber
\begin{split}
\langle A_1,A_1\rangle=a_1^2, \qquad&\langle A_2,A_2\rangle=\langle A_3,A_3\rangle=\frac 1{c_2^2},\\
\langle A_i,A_j\rangle=\langle \Gamma'(s),A_2\rangle=\langle \Gamma'(s),A_3\rangle=0\qquad&\mbox{ if  $i\neq j$}
\end{split}
\end{align}
and
\begin{equation}\nonumber
\langle \Gamma'(s),\Gamma'(s)\rangle-\frac{3 f'(s)}{2 f(s)^{7/4}}\langle A_1,\Gamma'(s)\rangle+\frac{9 f'(s)^2}{16 f(s)^{7/2}}\left(a_1^2+\frac1{c_2^2}\right)=1
\end{equation}
for a non-zero constant $a_1.$ Therefore, up to a suitable isometry of $\mathbb E^4$, we may assume
\begin{align}\label{MainTHMparametrizationEq5}
\begin{split}
A_1=&\left(0,0,a_1,0\right),\\
A_2=&\left(\frac 1{c_2},0,0,0\right),\\
A_3=&\left(0,\frac 1{c_2},0,0\right),\\
\Gamma(s)=&\left(0,0,\alpha_2(s)-\frac{a_1}{f(s)^{3/4}},\alpha_3(s)\right)
\end{split}
\end{align}
for some smooth functions $\alpha_2,\alpha_3$. Hence, \eqref{MainTHMparametrizationEq3} gives \eqref{MainTHMparametrizationSurfParam} after re-defining $t$ properly. Since $t=\mbox{const.}$ is an integral curve of $e_1$, Proposition \ref{PNMCVBicPrflCrv} implies that the curvature and torsion of the curve $\displaystyle\alpha(s)=\left(\frac1{c_2 f(s)^{3/4}},\alpha _2(s),\alpha _3(s)\right)$  are the functions $\kappa$ and $\tau$ given in \eqref{KTAUALL}.  
\end{proof}


Next, we obtain the converse of the above theorem.
\begin{theorem}\label{MainTHMparametrizationCnverse}
Let $M$ be the simple rotational surface in $\mathbb E^4$ given by
\eqref{MainTHMparametrizationSurfParam}
with arc-length parametrized smooth profile curve $\displaystyle\alpha(s)=\left(\frac1{c_2 f(s)^{3/4}},\alpha _2(s),\alpha _3(s)\right)$ whose curvature $\kappa $ is given by \eqref{KTAUALLa}, where $f:(a,b)\rightarrow\mathbb E^4$ is a positive function satisfying \eqref{mean curvature  partial differential} for a constant $c_2$. Then, $M$ is a PNMCV  biconservative surface. Furthermore, its mean curvature is  $f$.
\end{theorem}
\begin{proof}
Let $f$ be a positive function satisfying \eqref{mean curvature  partial differential} which is equivalent to
\begin{equation}\label{MCurvODE1}
f'(s)=\frac{4}{3} \epsilon  \sqrt{-c^2 f(s)^5+c_2^2 f(s)^{7/2}-9 f(s)^4}
\end{equation}
for a constant $\varepsilon\in\{-1,1\}$. Consider the curve $\alpha=(\alpha_1,\alpha_2,\alpha_3)$ with $\alpha_1=\frac{1}{c_2 f(s)^{3/4}}$, curvature $\kappa$ given by  \eqref{KTAUALLa} and torsion $\tau$. Let $t=(t_1,t_2,t_3)$, $n=(n_1,n_2,n_3)$, $b=(b_1,b_2,b_3)$ be the unit tangent, unit normal and unit binormal of $\alpha$. By a simple computation considering \eqref{MCurvODE1}, we obtain
\begin{subequations}\label{t1n1b1All}
\begin{eqnarray}
\label{t1n1b1Alla}t_1=\alpha_1'&=&-\frac{\epsilon  \sqrt{-c^2 f(s)^5+c_2^2 f(s)^{7/2}-9 f(s)^4}}{c_2 f(s)^{7/4}},\\
\label{t1n1b1Allb}n_1=\frac{t_1'}{\kappa}&=&\frac{f(s)^{1/4} \left(c^2 f(s)+3\right)}{c_2 \sqrt{c^2 f(s)+1}}.
\end{eqnarray}
Furthermore, $t_1^2+n_1^2+b_1^2=1$ and $\tau b=-n_1'-\kappa t$ give 
\begin{eqnarray}
\label{t1n1b1Allc}b_1&=&-\frac{2 c f(s)^{3/4}}{c_2 \sqrt{c^2 f(s)+1}},\\
\label{t1n1b1Alld}\tau&=&\frac{2 c \epsilon  \sqrt{c_2^2 f(s)^{5/2}-f(s)^3 \left(c^2 f(s)+9\right)}}{3 c^2 f(s)+3}.
\end{eqnarray}
\end{subequations}
By considering \eqref{MCurvODE1}, one can check that  \eqref{t1n1b1Alld}  is equivalent to \eqref{KTAUALLb} .

Now, let $M$ be the simple rotational surface in $\mathbb E^4$ given by
\eqref{MainTHMparametrizationSurfParam} with profile curve $\alpha$. Consider the local orthonormal frame field
$\{e_1,e_2,\hat e_3,\hat e_4\}$ on $M$ given by
\begin{eqnarray*}
e_1&=&\frac{\partial}{\partial s},\quad e_2=\frac{1}{\alpha(s)}\frac{\partial}{\partial t},\\
\hat e_3&=&\left( n_1(s)\cos t,n_1(s)\sin t,n_2(s),n_3(s)\right),\\
\hat e_4&=&\left( b_1(s)\cos t,b_1(s)\sin t,b_2(s),b_3(s)\right). 
\end{eqnarray*}

By a direct computation considering \eqref{t1n1b1Alla}-\eqref{t1n1b1Allc} and \eqref{KTAUALLb}, one can obtain \eqref{connectionsa}, \eqref{connectionsb} and
\begin{align}\label{PNMCVSurfAhate34}
\displaystyle
A_{\hat e_3}=\left(
\begin{array}{cc}
f(s) \sqrt{c^2 f(s)+1}&0\\
0&-\frac{f(s) \left(c^2 f(s)+3\right)}{\sqrt{c^2 f(s)+1}}
\end{array}
\right),\qquad A_{\hat e_4}=\left(
\begin{array}{cc}
0&0\\
0&\frac{2 c f(s)^{3/2}}{\sqrt{c^2 f(s)+1}}
\end{array}
\right).
\end{align}
Furthermore, the normal connection of $M$ satisfies
\begin{align}\label{PNMCVSurfDe2e3}
\begin{split}
\nabla^\perp_{e_2}{\hat e_3}=& 0,\\
\nabla^\perp_{e_1}{\hat e_3}=& \frac{2 c \epsilon  \sqrt{c_2^2 f(s)^{5/2}-f(s)^3 \left(c^2 f(s)+9\right)}}{3 c^2 f(s)+3}\hat e_4.
\end{split}
\end{align}
Note that the mean curvature vector of $M$ is 
$$H=-\frac{f(s)}{\sqrt{c^2 f(s)+1}}\hat e_3+ \frac{c f(s)^{3/2}}{\sqrt{c^2 f(s)+1}}\hat e_4$$
which yields that the mean curvature of $M$ is $H=f$. Thus, the normalized mean curvature vector of $M$ is
\begin{eqnarray}\label{PNMCVSurfNMCV}
e_3=-\frac{1}{\sqrt{c^2 f(s)+1}}\hat e_3+ \frac{c f(s)^{1/2}}{\sqrt{c^2 f(s)+1}}\hat e_4
\end{eqnarray}
and we put 
\begin{eqnarray}\label{PNMCVSurfNMCVe4}
e_4=\frac{c f(s)^{1/2}}{\sqrt{c^2 f(s)+1}}\hat e_3+ \frac{1}{\sqrt{c^2 f(s)+1}}\hat e_4
\end{eqnarray}
\eqref{PNMCVSurfAhate34}, \eqref{PNMCVSurfNMCV} and \eqref{PNMCVSurfNMCVe4} give \eqref{shape operator}. Furthermore, \eqref{connectionsc} follows from a direct computation considering \eqref{MCurvODE1}, \eqref{PNMCVSurfDe2e3} and \eqref{PNMCVSurfNMCV}. Thus, $M$ is a PNMCV surface and the orthonormal frame field $\{e_1,e_2;e_3,e_4\}$ satisfies conditions of Lemma \ref{Biconservative surfaces T.1.} which yields that $M$ is biconservative.
\end{proof}

In order to show the existence of PNMCV biconservative surfaces in $\mathbb E^4$, we would like to give the following example.
\begin{example}
Let $f$ be a positive function satisfying \eqref{mean curvature  partial differential} and assume that $f'$ does not vanish.
Then, the PNMCV simple rotational surface $M$  given by 
\begin{align}\label{MainTHMparametrizationSurfParamExample}
\begin{split}
x(s,t)=&\left(\frac{1}{c_2 f(s)^{3/4}}\cos t, \frac{1}{c_2 f(s)^{3/4}}\sin t,\frac1{c_2}\int_{s_0}^s{\cos (\theta (\xi)) f(\xi)^{1/4} \sqrt{c^2 f(\xi)+9} }d\xi \right.,\\&\left.
\frac1{c_2}\int_{s_0}^s \sin (\theta (\xi)) f(\xi)^{1/4}\sqrt{c^2 f(\xi)+9}  d\xi\right)
\end{split}
\end{align}
for a function $\theta$ satisfying
$$\theta '({s})=\frac{2 c c_2 f(s)^{5/4}}{c^2 f(s)+9}$$
is biconservative where $s_0,\; c\neq 0$ are some constants and $c_2$ is 
positive constant.
\end{example}

\section*{Acknowledgments}
This work was obtained during the project  supported by Research Fund of the Istanbul
Medeniyet University (Project Number: F-GAP-2017-986).

\end{document}